\documentclass{article}
\usepackage{amsmath} 
\usepackage{amsfonts}
\usepackage{amssymb, amsthm}

\newcommand{\ilim}{\mathop{\varprojlim}\limits}

\newtheorem{thm}{Theorem}

\newtheorem{cor}[thm]{Corollary}
\newtheorem{lem}[thm]{Lemma}

\newtheorem*{qu}{Question}

\newcommand{\R}{\mathbb{R}}
\newcommand{\N}{\mathbb{N}}

\begin{document}

\title{Spaces $\ell$-Dominated by $I$ or $\R$}

\author{Paul Gartside\footnote{\emph{Corresponding author} Department
    of Mathematics, University of Pittsburgh, Pittsburgh, PA 15260, USA, email:
    gartside@math.pitt.edu, tel: 412-624-7761, fax: 412-624-8397} and Ziqin Feng\footnote{Department of Mathematics and Statistics, Auburn University, Auburn, AL~36849, USA}}

\date{October  2015}

\maketitle

\begin{abstract} If $X$ is compact metrizable and has finite fd-height then the unit interval, $I$, $\ell$-dominates $X$, in other words, there is a continuous linear map of $C_p(I)$ onto $C_p(X)$. If the unit interval $\ell$-dominates a space $X$ then $X$ is compact metrizable and has countable fd-height. Similar results are given for spaces $\ell$-dominated by the reals.
\end{abstract}

\section{Introduction}

A space $X$ is said to be $\ell$--dominate another space $Y$
 if there is a linear continuous map of $C_p(X)$ onto $C_p(Y)$; and $X$ is
 $\ell$-equivalent to $Y$ if there is a linear homeomorphism
 of $C_p(X)$ and $C_p(Y)$. (For any space $Z$, $C(Z)$ is the set of all continuous real valued maps on $Z$, and $C_p(Z)$ is $C(Z)$ with the topology of pointwise convergence. All spaces are assumed to be Tychonoff.)
Extensive work has been done on the problems of which topological properties are preserved by $\ell$-equivalence, and which spaces are $\ell$-equivalent to standard spaces (see the surveys \cite{A_RP1, M_RP2}). For example in the first direction it is known that compactness, $\sigma$-compactness, hemicompactness, the $k_\omega$ property \cite{Ar}, and covering dimension \cite{Pest}, are preserved by $\ell$-equivalence, while results of Gorak \cite{Gorak} and Matrai \cite{Matrai} give a characterization of those spaces $\ell$-equivalent to the closed unit interval $I$. 
It seems natural to try to extend these results on $\ell$-equivalence to the context of $\ell$-dominance, in particular to determine which spaces are $\ell$-dominated by the closed unit interval. 
\begin{thm}\label{main1} \ 

(a) If $K$ is compact metrizable and has finite fd-height then the closed unit interval $\ell$-dominates $K$.  

(b) If a space $X$ is $\ell$-dominated by $I$ then $X$ is compact metrizable and has countable fd-height (equivalently, is strongly countable dimensional). 
\end{thm}
We use this result to derive an analogous theorem for spaces $\ell$-dominated by the reals, $\R$. Recall a space is submetrizable if it has a coarser metrizable topology. A $k_\omega$ sequence for a space $X$ is an increasing sequence of compact subsets $(K_n)_n$ of $X$ which is cofinal in the family of all compact subsets of $X$, and if $A$ is a subset of $X$ such that $A \cap K_n$ is closed in $K_n$ for all $n$, then $A$ is closed in $X$. A space is $k_\omega$ if it has a $k_\omega$ sequence. Note that $k_\omega$ spaces are $k$ spaces.

\begin{thm}\label{main2} \ 

(a) If $X$ is submetrizable and has a $k_\omega$ sequence of compact sets each of finite fd-height then the reals $\ell$-dominates $X$.  

(b) If a space $X$ is $\ell$-dominated by $\R$ then $X$ is submetrizable, has a  $k_\omega$ sequence and is strongly countable dimensional.
\end{thm}

Recall that a space  is strongly countable dimensional if it is the the countable union of closed finite dimensional subspaces. We define fd-height as follows. 
Let $K$ be a (non-empty) compact metrizable space.
Define $I(K) = \bigcup \{ U : U$ is open and finite dimensional$\}$, $K^{[0]}=K$, $K^{[\alpha +1]}\ = K^{[\alpha]} \setminus I(K^{[\alpha]})$, and $K^{[\lambda]} = \bigcap_{\alpha < \lambda} K^{[\alpha]}$ (for limit $\lambda$).
Note that $I(K)$ is always open, and, hence all the $K^{[\alpha]}$ are closed, and are well-ordered by reverse inclusion. By Baire Category $I(K)$ is non-empty if $K$ is strongly countable dimensional.  Further note that  $I(K)=K$ if and only if $K$ is finite dimensional.
From these observations we easily deduce that  $K$ is strongly countable dimensional if and only if $K^{[\alpha]} = \emptyset$ for some countable ordinal $\alpha$.
Further, when $K$ is strongly countable dimensional, there is a minimal $\alpha'$ so that $K^{[\alpha']} = \emptyset$. We call $\alpha'$ the \emph{fd-height} of $K$. For limit $\lambda$, $K^{[\lambda]}$ cannot be empty, so  $\alpha'$ is always a successor, $\alpha' = \alpha +1$, and  $K^{[\alpha]}$ is compact metrizable and finite dimensional.

We remark that there are compact metrizable spaces with  fd-height $k$ for every natural number $k$. Indeed, let $F_1$ be any finite dimensional compact metrizable space ($I^n$, or just a single point), then $F_1$ has fd-height $1$. The one point compactification, $F_2$, of $\bigoplus_n I^n$, has fd-height $2$. Inductively, given $F_k$, a compact metrizable space of fd-height $k \ge 2$, it is well known that there is a metrizable compactification of $\N$ with remainder $F_k$, replacing each isolated point in this compactification with a copy of $F_2$, gives a compact metrizable space  $F_{k+1}$ with fd-height $k+1$.

Part~(a) of Theorem~\ref{main1} is established in Section~\ref{Ildom}.
The proof that for every compact, metrizable space $K$ with finite fd-height there is a continuous linear map of $C_p(I)$ onto $C_p(K)$ proceeds by  induction on fd-height. To keep the induction going the linear maps need to have the stronger property of being $c$-good (defined and discussed below at the start of Section~\ref{Ildom}). In order to start (and end) the induction, $c$-good linear maps of $C_p(I)$ onto $C_p(K)$ are needed for \emph{finite dimensional} compact metrizable  spaces $K$. These come from Kolmogorov's solution of Hilbert's 13th Problem, which asks whether every continuous function of three variables can be written as a composition of continuous functions of two variables). Kolmogorov proved that for every $n \ge 2$ there are maps $\phi_1, \ldots , \phi_r$ in $C(I^n)$ so that every $f$ in $C(I^n)$ can be written in the form $f = \sum_i g_i \circ \phi_i$ of some $g_i$ in $C(\R)$, and further the $\phi_i$ can be written as a sum of functions of one variable. It is easy to see that the $g_i$ can be chosen to coincide. Then the map $g \mapsto \sum_i g \circ \phi_i$ is a continuous linear map of $C_p(I)$ onto $C_p(I^n)$. With a sufficiently careful choice of the $\phi_i$ this map can be made to be $c$-good (Theorem~\ref{cg}). The authors have  shown \cite{FG2} that if a Kolmogorov type theorem holds for a space $X$ then that space must be locally compact, separable metrizable and finite dimensional. Thus the linear maps witnessing part~(a) of  Theorem~\ref{main1} (and Theorem~\ref{main2}) cannot come from a Kolmogorov type theorem.

Part (a) of Theorem~\ref{main2} is established in Section~\ref{Rldom} by careful gluing of continuous linear surjections coming from part~(a) of  Theorem~\ref{main1}.

In Section~\ref{Pldom} we establish part~(b) for Theorem~\ref{main2}. Our results also prove part~(b) of Theorem~\ref{main1}, but  Leiderman et al proved this in \cite{LMP}. 
Baars \cite{Baars} showed that if $X$ $\ell$-dominates $Y$ and $X$ is ($\sigma$-)compact then so is $Y$ provided $Y$ is a $\mu$-space (bounded subsets are closed). We show (Lemma~\ref{ld_imp}) that the same is true for hemicompactness and $k_\omega$, under the same assumption on $Y$. We deduce (Lemma~\ref{ld_kwsbscd}) that if $X$ $\ell$-dominates $Y$ and $X$ is $k_\omega$, submetrizable and strongly countable dimensional then so is $Y$ (with no restriction on $Y$). Since both $\R$ and $I$ are $k_\omega$, submetrizable and strongly countable dimensional, and $I$ is additionally compact, part (b)  of  Theorems~\ref{main1} and~\ref{main2} both hold. 

There remain two striking open questions:
\begin{qu} \ 

(1) If $K$ is compact metrizable and strongly countable dimensional then is $K$ $\ell$-dominated by $I$? 

(2) If $K$ is $\ell$-dominated by $I$ then does $K$ have finite fd-height?
\end{qu}
(Similar questions also pertain to spaces $\ell$-dominated by the reals, but their solution should follow from the case for $I$ combined with Theorem~\ref{key2}.)

In the proof of part~(a) of Theorem~\ref{main1} the constant $c$ involved in the relevant `$c$-good map' (see discussion above) increases unboundedly with the fd-height. This prevents the induction continuing transfinitely. The authors do not see how to avoid this, so we tentatively conjecture that Question (2) has a positive answer.

\section{Constructing Spaces $\ell$-dominated by $I$}\label{Ildom}

\subsection{$c$--Good and Other Linear Maps}
Let $K$ be a compact metrizable space. Note that in addition to the topology of pointwise convergence, $C(K)$ carries the 
supremum norm, $\|\cdot\|$. 
For any closed subspace $A$ of $X$ write $C(K;A) = \{ f \in C(K) : f \restriction A =0\}$. Of course $C(K;\emptyset)=C(K)$. Note that $C(X;A)$ is a closed subspace of $C_p(X)$, and hence is closed with respect to the supremum norm.  

Let $K_1$ and $K_2$ be compact, and let $A_1$ and $A_2$ be closed subsets of $K_1$ and $K_2$, respectively. Let $L$ a map of $C_p(K_1;A_1)$ to $C_p(K_2;A_1)$. Then $L$ is \emph{$c$-good} if $L$ is a continuous linear map, $\| L \| \le 1$  and for every $f$ in $C(K_2;A_2)$ there is a $g \in C(K_1;A_1)$ such that $L(g)=f$ and $\| g \| \le c . \|f\|$. We call the last condition the `small kernel' property. Obviously the small kernel property for $L$ implies $L$ is a surjection.
An easy calculation gives:
\begin{lem}\label{comp_cgood}
For $i=1,2$ let $L_i : C_p(K_i;A_i) \to C_p(K_{i+1};A_{i+1})$ be $c_i$-good. Then $(L_2 \circ L_1) : C_p(K_1;A_1) \to C_p(K_3;A_3)$ is $(c_1 c_2)$-good.
\end{lem}

We will also need Dugundji's Extension Theorem:
\begin{lem}[Dugundji]\label{ext_map} Let $A$ be a closed subspace of metrizable $X$. 
There is a map $e : C_p(A) \to C_p(X)$ such that $e(f)\restriction A
=f$, and such that $e$ is linear, pointwise continuous, $\|e\|=1$
and respects the convex hull.
\end{lem}
It follows that $c$-good maps between compact metrizable spaces are well behaved under restriction in the range and extension in the domain.
\begin{lem}\label{cl_sub}
If $A_1 \subseteq K_1 \subseteq K_1'$ and $A_2 \subseteq K_2' \subseteq K_2$ where all spaces are compact metrizable, and $L: C_p(K_1,A_1) \to C_p(K_2,A_2)$ is $c$-good, then $L' : C_p(K_1';A_1) \to C_p(K_2';A_2)$ is also $c$-good  where $L'(g) = L(g\restriction K_1) \restriction K_2'$.
\end{lem}
\begin{proof} 
As $L$ is linear, $L'$ is clearly linear. Since $\|L\|\leq 1$, we have that $\|L'\|\leq 1$ also.
As $L$ and restriction maps are continuous, the map $L'$ is continuous. 

We verify the `small kernel' property for $L'$. First fix maps $e_1: C_p(K_1)\rightarrow C_p(K_1')$ and $e_2: C_p(K_2')\rightarrow C_p(K_2)$ with the properties guaranteed by Dugundji's Extension Theorem (Lemma~\ref{ext_map}). 
Take $f\in C_p(K_2';A_2)$. Then $e_2(f)\in C_p(K_2; A_2)$ and $\|e_2(f)\|\leq \|f\|$. Since $L$ is $c$-good there exists $g\in C_p(K_1, A_1)$ such that $L(g)=e_2(f)$ and $\|g\|\leq c.\|f\|$. Then we see that $e_1(g)$ is in $C_p(K_1';A_1)$ and $L'(e_1(g))=f$. Also $\|e_1(g)\|\leq \|g\|\leq c.\|e_2(f)\|\leq c.\|f\|$. Hence $L'$ is $c$-good.  \end{proof}

\subsection{Constructing $c$-Good Linear Surjections}

\paragraph{Basic Families Give $c$-Good Maps}

\begin{thm}\label{cg} For any $c>1$, and for each $n$, there is a $c$--good map $L_n : C_p(I) \to C_p(I^n)$.
\end{thm}

\begin{proof}
The argument uses results of Sprecher as presented in Braun \& Griebel \cite{BG}.

Fix $n$. Let $m$ be such that $(2n+1)/(m+1) \le 1-1/c = \eta$. Note that $0<\eta <1$. Let $\epsilon = 1/(m-n-1)$.
Then $\frac{m-n-1}{m+1} \epsilon + \frac{2n}{m+1} = \frac{2n+1}{m+1} \le \eta$.

So according to \cite{BG} we get  $\phi_0, \ldots , \phi_m$ from $C(I^n)$ with the following properties: given $f_0 \in C(I^n)$ there are $g_q^k$ in $C(I)$ so that
\begin{itemize}
\item $g_q= \sum_{k=1}^\infty g_q^k$ is in $C(I)$,
\item  $f_0 = \sum_{q=0}^m g_q \circ \phi_q$,
\item $\|g_q^k\| \le \frac{1}{m+1}\|f_{k-1}\|$, where for $k \ge 1$, $f_k = f_0-\sum_{q=0}^m \sum_{j=1}^k g_q^k \circ \phi_q$ (see Lemma~3.8 of \cite{BG}), and
\item $\|f_k\| \le \eta \|f_{k-1}\|$ for $k \ge 1$ (see Theorem~3.3 of \cite{BG}).
\end{itemize}
(Note that Braun \& Griebel use $r$ instead of our $k$, $\Phi_q^r$ for $g_q^k$, and $\xi_q$ for $\phi_q$.)

Pick pairwise disjoint closed non--trivial subintervals $I_0, \ldots , I_m$ inside $I$. Pick a homeomorphism $h_q : I \to I_q$. Let $\hat{\phi}_q = h_q \circ \phi_q$.
Define $L_n =L: C_p(I) \to C_p(I^n)$ by
$ L(g) = \frac{1}{m+1} \sum_{q=0}^m (g \restriction I_q) \circ \hat{\phi}_q$.

We show $L$ is $c$-good in four steps. Clearly,  (1) $L$ is continuous and linear. And, (2)  $\|L\| \le 1$ because:
\[ \|L(g)\| \le \frac{1}{m+1} \sum_{q=0}^m \| (g \restriction I_q)\circ \hat{\phi}_q\| \le \frac{1}{m+1} (m+1) \|g\| = \|g\|.\]

Next we show (3) $L$ is onto. Take any $f$ in $C(I^n)$. Let $f_0=(m+1)f$. Apply the Sprecher result to $f_0$ to get $g_q$'s  as above, so
$f_0 = \sum_{q=0}^m g_q \circ \phi_q = \sum_{q=0}^m (g_q \circ h_q^{-1}) \circ \hat{\phi}_q$.
Hence $f = \frac{1}{m+1} \sum_{q=0}^m (g \restriction I_q) \circ \hat{\phi}_q = L(g)$, where $g \in C(I)$ is $g_q \circ h_q^{-1}$ on $I_q$ and otherwise interpolated linearly over $I$.

It remains to show (4) that $L$ has the `small kernel' property. Take $f$ in $C(I^n)$ and $g$ as in (3).
Note that $\| g\| \le \max \|g_q\|$. We show that the $g_q$ satisfy $\|g_q\| \le c \|f\|$, as required for `small kernel'.
Each $g_q = \sum_{k=1}^\infty g_q^k$ where $\| g_q^k \| \le \frac{1}{m+1} \|f_{k-1}\|$. Since $\|f_{k}\| \le \eta \|f_{k-1}\|$, inductively we see that $\|g_q^k\| \le \frac{1}{m+1} \eta^{k-1} \|f_0\| = \eta^{k-1} \|f\|$, and so
\[ \|g_q\| \le \sum_{k=1}^\infty \|g_q^k\| \le \sum_{k=0}^\infty \eta^k \|f\| = \frac{1}{1-\eta} \|f\| = c \|f\|.\]
\end{proof}

\begin{cor}\label{cptfd_cgood}
For every $c>1$ and compact, metrizable finite dimensional space $K$, there is a $c$-good map of $C_p(I)$ onto $C_p(K)$
\end{cor}
If $K$ has dimension $n$, then it embeds as a closed subspace in $I^{2n+1}$. Hence we can apply Theorem~\ref{cg} and Lemma~\ref{cl_sub}.

\paragraph{$c$-Good Maps Exist For Large Enough $c$}

\begin{thm}\label{key} Let $K$ be compact metrizable and have finite fd-height, $k'=k+1$. Then there is a $(8^{k+1})$-good map of $C_p(I)$ to $C_p(K)$.

In particular, the closed unit interval $\ell$-dominates $K$.
\end{thm}

\begin{proof} We proceed by induction on the fd-height, $k'$.

\paragraph{Base Case: $K^{[0+1]}= \emptyset$.} Then $K$ is compact metrizable and finite dimensional, and so has (Corollary~\ref{cptfd_cgood})  an  $8$-good map.

\paragraph{Inductive Step: $K^{[k+1]}= \emptyset$.}
 Then $K^{[k]}$ is compact metrizable and finite dimensional. Assume the claim is true for all compact metrizable spaces with fd-height $\le k$.

Fix an extender $e$ of $K^{[k]}$ in $K$. The map $H: C_p(K; K^{[k]}) \times C_p(K^{[k]}) \to C_p(K)$ defined by $H(g,h) = g+ e(h)$, is an isomorphism. 
Fix a $2$-good map $L_2 : C_p([2/3,1]) \to C_p(K^{[k]})$. We show there is a $2.8^{k}$-good map $L_1 : C_p([0,1/3]) \to C_p(K; K^{[k]})$, for then $L : C_p(I) \to C_p(K)$ is $8^{k+1}$-good where $L(g) = \frac{1}{2} \cdot H(L_1(g \restriction [0,1/3]), L_2 ( g \restriction [2/3,1]))$. 
To see this, first observe that $L$ is certainly continuous, linear and has norm $1$. We show it has the `small kernel' property. 
Take any $f$ in $C_p(K)$. Let $\hat{f}=2f$. Let $f_1=\hat{f} - e(\hat{f} \restriction K^{[k]})$ and $f_2=\hat{f} \restriction K^{[k]}$. Note that $\|f_1\| \le 2\|\hat{f}\|$ and $\|f_2\| \le \|\hat{f}\|$. 
As $L_1$ is $2.8^k$-good there is a $g_1$ such that $L_1(g_1)=f_1$ and $\|g_1\| \le 2.8^k \|f_1\|$. 
As $L_2$ is $2$-good there is a $g_2$ such that $L_2(g_2)=f_2$ and $\|g_2\| \le 2 \|f_1\|$. Let $g$ in $C_p(I)$ be the continuous function which is $g_1$ on $[0,1/3]$, $g_2$ on $[2/3,1]$, and linearly interpolates between $g_1(1/3)$ and $g_2(2/3)$ on $[1/3,2/3]$. Then $L(g)=(1/2)(f_1+e(f_1)) = (1/2) (\hat{f}-e(\hat{f} \restriction K^{[k]})+ e(\hat{f}\restriction K^{[k]})) =f$. And, $\|g\|=\max (\|g_1\|,\|g_2\|) \le \max (2.8^k \|f_1\|,2 \|f_2\|) \le 2.8^{k}. 2\|\hat{f}\| \le 2.8^k.2.2 \|f\| = 8^{k+1} \|f\|$.

The map $g \mapsto (1/2)(g-c_{g(0)})$, where $c_t$ is the constant $t$ function, from $C_p(I)$ to $C_p(I;\{0\})$ is clearly $2$-good. So it suffices,  by Lemma~\ref{comp_cgood}, to find  an $8^k$-good map $L_0 : C_p(I; \{0\}) \to C_p(K; K^{[k]})$.

Let $U = K \setminus K^{[k]}$. Find a countable locally finite (in $U$) open cover of $U$, $\mathcal{U} = \{U_n \}_n$, where each closed set $C_n = \overline{U_n}$ is contained in $U$, and a partition of unity $\{p_n\}_n$, (so each $p_n$ is continuous, maps into $[0,1]$ and $\sum_n p_n =1$), such that the support of $p_n$ is contained in $U_n$. We can suppose each $p_n$ is in $C_p(K; K^{[k]})$ (and not just $C_p(U)$).

Fix for the moment $n$. Then $C_n$ is compact metrizable and has fd-height $\le k$. So by inductive hypothesis there is a $8^{k}$-good map $l_n' : C_p(I_n) \to C_p(C_n)$ where $I_n = [(3/2)(1/2^n),1/2^{n-1}]$. Fix an extender $e_n$ for $C_n$ in $K$. Let $l_n = e_n \circ l_n' : C_p(I_n) \to C_p(K)$. Note that if $g_n$ is in $C_p(I_n)$ then $l_n(g_n).p_n$ is equal to $l_n'(g_n).p_n$ on $C_n$, and is zero outside $C_n$.

Now we can define $L_0' : C_p( \bigoplus_n I_n \cup \{0\}; \{0\}) \to C_p(K; K^{[k]})$ by $L_0'(g) = \sum_n l_n (g_n). p_n$ where $g_n = g \restriction I_n$. Define $L_0$  by $L_0(g) = L_0'(g \restriction \left( \bigoplus_n I_n \cup \{0\} \right))$.
Once we have  checked that $L_0'$ is well-defined and $8^k$-good, then $L_0$ is as required by Lemma~\ref{cl_sub}. 

Certainly $L_0'(g)$ exists, maps $K$ into $\R$ and is zero on $K^{[k]}$. Also $L_0'(g)$ is continuous: if $g \in C_p(\bigoplus_n I_n \cup \{0\};\{0\})$ then $\|g_n\| \to 0$, so $\|l_n(g_n)\| \to 0$, and -- remembering that the $p_n$'s form a partition of unity -- continuity of $L_0'(g)$ follows.

Since the $l_n$'s are linear, so is $L_0'$.
Next we check $L_0'$ is continuous. It suffices to check continuity at $\mathbf{0}$. Take subbasic $B_\epsilon (\mathbf{0}, \{x_0\}) = \{ f \in C_p(K;K^{[k]}) : |f(x_0)| < \epsilon\}$. If $x_0$ is in $K^{[k]}$ there is nothing to check. So suppose $x_0$ is in $U$. Then there is an open $V \ni x_0$ such that $V$ meets $U_{n_1}, \ldots , U_{n_r}$ from $\mathcal{U}$, and no others. Note $x_0 \notin C_n = \overline{U_n}$ if $n \ne n_i$ for some $i$.
For $i=1, \ldots , r$, as $l_{n_i}' : C_p(I_{n_i}) \to C_p(C_{n_i})$ is  $8^k$-good, there is $\delta_i >0$ and finite $F_i$ such that $l_{n_i} ( B_{\delta_i} (\mathbf{0}, F_i)) \subseteq B_\epsilon (\mathbf{0}, \{x_0\})$. Then $\delta = \min_i \delta_i$ and $F=\bigcup_i F_i$ work for $L_0'$. To see this, take $g\in B_{\delta}(\mathbf{0}, F)\cap C_p( \bigoplus_n I_n \cup \{0\}; \{0\})$ and let $g_n=g\restriction I_n$. We see that $g_{n_i}\in B_{\delta_i}(\mathbf{0}, F_i)$ for each $i=1, \ldots, r$, hence $|\ell_{n_i}(g_{n_i})(x_0)|<\epsilon$ for each $i=1, \ldots, r$. Also, since $x_0\notin C_n$ if $n\ne n_i$ for some $i$, $p_{n}(x_0)=0$ if $n\ne n_i$ for some $i$. Therefore, we have that $|L_0'(g)(x_0)|=|\sum_n \ell_n(g_n)(x_0).p_n(x_0)|\leq \sum_{i=1}^{r}|\ell_{n_i}(g_{n_i})(x_0).p_{n_i}(x_0)| <\epsilon$.

Since each $l_n$ has norm $\le 1$, so does $L_0'$. 
Finally, take any $f$ in $C_p(K; K^{[k]})$. For each $n$ let $f_n= f \restriction C_n$. Since $f \restriction K^{[k]}$ is identically zero, $\|f_n\| \to 0$. As $l_n'$ is $8^k$-good, we can pick $g_n$ in $C_p(I_n)$ such that $l_n'(g_n)=f_n$ and $\| g_n\| \le (8^k) \|f_n\|$.
Let $g$ be the function which coincides with $g_n$ on $I_n$, is zero at zero, and otherwise linearly interpolate between the endpoints of the $I_n$'s. Since $\|g_n\| \le 8^k \|f_n\|$ and $\|f_n\| \to 0$ the  function $g$ is continuous. (This is the critical use of the `small kernel' property. Without it the $g_n$'s could have  norms not converging to zero, and $g$ would not be continuous.)  Now $L_0'(g)=f$ (here we use the fact that the $p_n$'s are a partition of unity, $C_n$ is contained in the support of $p_n$ and $l_n(g_n).p_n$ coincides with $l_n'(g_n).p_n$ on $C_n$) and $\| g \| \le (8^k) \|f\|$.
\end{proof}

\section{Constructing Spaces $\ell$--Dominated by $\R$}\label{Rldom}

Since every compact metrizable space of finite fd-height is $\ell$-dominated by $I$, part~(a) of Theorem~\ref{main2} follows immediately from the next result.

\begin{thm}\label{key2}
The reals, $\R$, $\ell$--dominates a Tychonoff space $X$
if  $X$ is submetrizable and has a $k_\omega$ sequence, $(K_n)_n$ where each $K_n$ is $\ell$-dominated by $I$.
\end{thm}

\begin{proof} 
Suppose $X$ is submetrizable and has a $k_\omega$ sequence, $(K_n)_n$ where each $K_n$ is $\ell$-dominated by $I$.  Hence, $C_p(X)$ is linearly homeomorphic to $\ilim C_p(K_{n})$, where the $K_n$'s are directed by inclusion, and $f \in C_p(X)$ corresponds to $(f\restriction K_n)_n$ in the inverse limit. Let $\pi_n$ be the projection map from $\ilim C_p(K_{m})$ to $C_p(K_n)$.

Fix $\{I_n\}_n$  a discrete family of  closed non--trivial intervals in
$\R$. Let $\rho_n : C_p(\R) \to C_p(I_n)$ be the restriction map, $\rho_n (g)=g \restriction I_n$.  This is, of course, continuous linear and onto. By Lemma~\ref{ext_map}, pick $e_n^{n+1}$  a continuous,
linear extender from $C_p(K_n)$ to $C_p(K_{n+1})$. 

By hypothesis (all the $K_n$'s are $\ell$-dominated by $I$), there is a continuous linear
onto mapping $L_1:C_p(I_1)\rightarrow C_p(K_1)$. Take any $n \ge 2$.
Recall that (using the extender $e_{n-1}^n$) $C_p(K_n)$ is linearly homeomorphic to $C_p(K_n;K_{n-1})\times C_p(K_{n-1})$. Hence the composition of the hypothesized continuous, linear surjection of $C_p(I_n)$ to $C_p(K_n)$, with the projection onto the first factor of the above product, is a continuous, linear and onto mapping
$L_n:C_{p}(I_n)\rightarrow  C_p(K_n;K_{n-1})$.

Define $L:C_p(\R)\rightarrow \ilim C_p(K_n)$ as
follows. We need to specify $\pi_n \circ L$ from $C_p(\R)$ to  $C_p(K_n)$ for each $n$. Recursively set  $\pi_1 \circ L=L_1 \circ \rho_1$ and assuming that  $\pi_n \circ L$ is
defined, set $\pi_{n+1} \circ L=e_n^{n+1} \circ (\pi_n \circ L)+L_{n+1} \circ \rho_{n+1}$.
Note that for every $g$ in $C_p(\R)$ and every $n$,  $(\pi_n \circ L)(g)$ is a continuous map on 
$K_{n}$, and for $n \ge 2$,  extends $(\pi_{n-1} \circ L)(g)$. Hence $L$ is well defined.

Now we need to prove that $L$ is continuous, linear and onto. For continuity and linearity, it is sufficient to verify that each $\pi_n\circ L: C_p(\R)\rightarrow
C_p(K_n)$ is continuous and linear.
This is straightforward by induction. Since $L_1$ and $\rho_1$ are continuous and linear so is $\pi_1\circ L=L_1 \circ \rho_1$.
Suppose inductively that $\pi_n\circ L$ is continuous and linear.
By definition, $\pi_{n+1}\circ L=e_n^{n+1} \circ (\pi_n \circ L)+L_{n+1} \circ \rho_{n+1}$ is a sum of compositions of continuous and linear maps, and so has the same properties.

It remains to show that $L$ is onto.  Take any
$(f_n)_n\in \ilim C_p(K_{n})$ (so each $f_n$ is continuous on $K_n$ and $f_{n+1}$ extends $f_n$). We inductively find a sequence $(g_n)_n$ where $g_n \in C_p(I_n)$ such that any (and there are many such) $g$ in $C(\R)$ with $\rho_n(g)=g_n$ for all $n$, satisfies $(\pi_n \circ L)(g) = f_n$ for all $n$, in other words $L(g)=(f_n)_n$.

To start the induction, as $L_1$ is onto, we can pick $g_1$ in $C_p(I_1)$ so that $L_1(g_1)=f_1$. Note that for any $g$ in $C_p(\R)$ such that $\rho_1(g)=g_1$ we have $(\pi_1 \circ L)(g)=f_1$.

Inductively suppose we have $g_1, \ldots , g_n$ where $g_i \in C_p(I_i)$ are such that any $g$ in $C_p(\R)$ with $\rho_i(g)=g_i$ for $i=1, \ldots ,n$ satisfies $(\pi_n \circ L)(g) = f_n$. As $L_{n+1}$ is onto there is a $g_{n+1}$ in $C_p(I_{n+1})$ such that $L_{n+1}(g_{n+1}) = f_{n+1}-e_n^{n+1}(f_{n+1} \restriction K_n)$. Now if $g \in C_p(\R)$ is such that $\rho_i(g)=g_i$ for $i=1, \ldots , n+1$ then $(\pi_{n+1} \circ L)(g) = e_n^{n+1}( (\pi_n \circ L) (g)) +L_{n+1}(g \restriction I_{n+1}) = e_n^{n+1}(f_n)+L_{n+1}(g_{n+1}) = f_{n+1}$ --- as required.
\end{proof}

\section{Properties Preserved by $\ell$-Domination}\label{Pldom}

The main aim here is to show that a space $\ell$-dominated by $\R$ is
$k_\omega$, submetrizable and   strongly countable dimensional.

In \cite{Baars} Baars showed that if  a ($\sigma$-) compact
space $X$ $\ell$-dominates $Y$, then $Y$ is ($\sigma$-)
compact if and only if $Y$ is a $\mu$-space (closed bounded
subsets are compact). We establish similar results for hemicompact
and for $k_\omega$ spaces.

\begin{lem}\label{ld_imp} Suppose $X$ $\ell$-dominates  $Y$ via $L$. Then
\begin{itemize}
\item[(i)] if $X$ is hemicompact then $Y$ is hemicompact if and only if $Y$ is $\mu$,
\item[(ii)] if $X$ is $k_\omega$ then $Y$ is $k_\omega$ if and only if $Y$ is $\mu$.
\end{itemize}
\end{lem}

\begin{proof}
First note that hemicompact spaces are normal, and so $\mu$. Suppose now that $X$ is hemicompact, and $\ell$-dominates a $\mu$ space $Y$ by $L$. The linear map $L$ gives rise to the support map $y \mapsto \mathop{supp} (y)$ where $\mathop{supp} (y)$ is a finite subset of $X$ whose most relevant properties here are that (1) if $A$ is a bounded subset of $Y$ then $\mathop{supp} (A) = \bigcup \{ \mathop{supp} (y) : y \in A\}$   is a bounded subset of $X$ (Arhangelskii \cite{Ar}) and (2) if $B$ is a closed bounded subset of $X$ then $\mathop{supp}^{-1} B = \{ y : \mathop{supp} (y) \subseteq B\}$ is a closed bounded subset of $Y$ (Baars \cite{Baars}).

Let $K_n$ be be a countable cofinal sequence of compact sets in $X$. Define $K'_n=\mathop{supp}^{-1} K_n$. As $Y$ is $\mu$ these $K'_n$'s are compact. If $K'$ is compact in $Y$, then there is an $N$ so that the compact set $\overline{\mathop{supp} (K')}$ is contained in $K_N$. It is easy to check that $K' \subseteq K'_N$ -- so the $K'_n$'s are cofinal in the compact subsets of $Y$.

Finally suppose $X$ is $k_\omega$ and $Y$ is $\mu$. Since $X$ is hemicompact, so is $Y$, Thus $C_k(Y)$ is metrizable, and $Y$ is Dieudonne complete, so  $L$ is a continuous linear map of $C_k(X)$ onto $C_k(Y)$. Since $Y$ is $\mu$, $C_k(Y)$ is barrelled (the Nachbin-Shirota theorem \cite{Nach, Shir}). Recall \cite{Warner} that $C_k(Z)$ is completely metrizable if and only if $Z$ is $k_\omega$. Thus $C_k(X)$ is a completely metrizable locally convex topological vector space, and it suffices to show that $C_k(Y)$ is also completely metrizable.

Let $\tau_k$ be the compact-open topology on $C(Y)$. Let $\tau_L$ be the quotient topology on $C(Y)$ induced by $L$. Then $(C(Y),\tau_L)$ is a completely metrizable locally convex space, because $C_k(X)$ has these same properties. We show that the identity map $i:(C(Y),\tau_k) \to (C(Y),\tau_L)$ is continuous, so the two topologies coincide. First note that $i$ is almost-continuous: take any $\tau_L$-neighborhood $U$ of $\mathbf{0}$, which by local convexity we can assume is a $\tau_L$-closed barrel, then the $\tau_k$-closure of $U$ is a $\tau_k$-closed barrel, and hence (as $C_k(Y)$ is barrelled) contains a $\tau_k$-neighborhood $V$ of $\mathbf{0}$ such that  $V \subseteq \overline{U}^{\tau_k} = \overline{i^{-1} U}^{\tau_k}$. Now $i$ is linear, almost-continuous and has a closed graph, with range a completely metrizable linear space, and so by a standard Closed Mapping Theorem is continuous (see 14.3.4 in \cite{NariciBeck} for example).
\end{proof}

\begin{lem}\label{ld_kwsbscd}
If $X$ is a $k_\omega$, submetrizable space and is strongly countable dimensional, and $X$ $\ell$-dominates a space $Y$ via $L$, then $Y$ is also a $k_\omega$, strongly countable dimensional,  submetrizable space.
\end{lem}

\begin{proof}
As $X$ $\sigma$-compact it is submetrizable if and only if it is cosmic. So  $C_p(X)$ is cosmic, and hence so is $C_p(Y)$ and $Y$. In particular $Y$ is $\mu$.

Hence, as $X$ is $k_\omega$ from the preceding lemma, $Y$ is $k_\omega$ and submetrizable, and $L$ is open and continuous as a map of $C_k(X)$ onto $C_k(Y)$. As $X$ and $Y$ are $k$-spaces, the free locally convex topological spaces on $X$ and $Y$ both get their topologies as the dual space of $C_k(X)$ and $C_k(Y)$ respectively. Thus $Y$ embeds as a closed subspace in $L(Y)$, which embeds as a closed subspace in $L(X)$. Also $L(X)$ is a countable union of closed subspaces homeomorphic to closed subspaces of spaces of the form $I^n \times X^n$ for some $n$, so as $X$ is strongly countable dimensional, so is $L(X)$, and hence $Y$.
\end{proof}

\end{document}